\documentclass[oneside, english]{article}
\usepackage[T1]{fontenc}
\usepackage[latin1]{inputenc}
\usepackage{amssymb}
\usepackage{amsmath}
\usepackage[all]{xy}
\usepackage{amsthm}
\usepackage{setspace}

\makeatletter

  \theoremstyle{definition}
  \newtheorem{defn}{Definition}[section]

  \theoremstyle{remark}
    
   \newtheorem{remark}[defn]{Remark}

   \theoremstyle{plain}
  \newtheorem{theorem}[defn]{Theorem}
  \newtheorem{lemma}[defn]{Lemma}

  \newtheorem{corollary}[defn]{Corollary}

    \newtheorem{example}[defn]{Example}

\usepackage{geometry}

\usepackage{babel}

\usepackage{amsmath,amssymb,amsfonts,latexsym}
\usepackage{mathrsfs}
\usepackage{epsf}
\usepackage{epsfig}
\usepackage[all]{xy}
\usepackage{enumerate}
\makeatother

\begin{document}

\def\AW[#1]^#2_#3{\ar@{-^>}@<.5ex>[#1]^{#2} \ar@{_<-}@<-.5ex>[#1]_{#3}}
 \def\NAW[#1]{\ar@{-^>}@<.5ex>[#1] \ar@{_<-}@<-.5ex>[#1]}

 \renewcommand{\arraystretch}{1}
 \renewcommand{\O}{\bigcirc}
 \newcommand{\OX}{\bigotimes}
 \newcommand{\OD}{\bigodot}
 \newcommand{\OV}{\O\llap{v\hspace{.6ex}}}
 \newcommand{\B}{\mbox{\Huge $\bullet$}}
  \newcommand{\D}{$\diamond$}

\title{Weight modules of $D(2,1,\alpha)$}
\author{Crystal Hoyt\footnote{Department of Mathematics, Technion - Israel Institute of Technology, Haifa 32000, Israel; hoyt@tx.technion.ac.il.}}
\date{July 30, 2013}
\maketitle

\begin{abstract}
Let $\mathfrak{g}$ be a basic Lie superalgebra.  A weight module $M$ over $\mathfrak{g}$ is called
finite if all of its weight spaces are finite dimensional, and it is
called bounded if there is a uniform bound on the dimension of a
weight space.  The minimum bound is called the degree of $M$.
For $\mathfrak{g} = D(2,1,\alpha)$, we prove that every simple weight module $M$ is bounded and has degree less than or equal to $8$. This bound is attained by a cuspidal module $M$ if and only if $M$ belongs to a $(\mathfrak{g},\mathfrak{g}_{\bar{0}})$-coherent family  $L(\lambda)_{\Gamma}^{\mu}$ for some typical module $L(\lambda)$. Cuspidal modules which correspond to atypical modules have degree less than or equal to $6$ and greater than or equal to $2$.
\end{abstract}

\section{Introduction}
Basic Lie superalgebras are a natural generalization of simple finite dimensional Lie algebras, and their
finite dimensional modules have been studied extensively \cite{K77, M12}. In this case, every simple $\mathfrak{g}$-module is a highest weight module with respect to each choice of simple roots of $\mathfrak{g}$, and moreover, there exist various character formulas which help one ``count'' the dimension of each weight space of the module.
A natural generalization of this setting is to the study of (possibly) infinite dimensional modules that have finite dimensional weight spaces, namely, finite weight modules.  However, since these modules are not necessarily highest weight with respect to any choice of the set of simple roots, the question arises how to characterize and classify all such simple modules.

In \cite{M00}, Mathieu gave an answer to this problem for simple Lie algebras
by relating to each simple finite weight module $M$ a corresponding simple highest weight module $L(\lambda)$ such that $M$ is the ``twisted localization'' of $L(\lambda)$, that is, $M\cong L(\lambda)_{\Gamma}^{\mu}$.
Grantcharov extended this result to cover classical Lie superalgebras in \cite{Gr09}.
Using this characterization one can gain information about a simple finite weight module $M$ from the corresponding simple highest weight module $L(\lambda)$, including the calculation of its degree (see (\ref{degree})).
Moreover, this is a major step towards the classification of simple finite weight modules.

In addition, one must determine which simple highest weight modules $L(\lambda)$ can appear in this correspondence. These are the modules which have uniformly bounded weight multiplicities, the so called ``bounded modules''.  Then one should determine the simplicity conditions for the modules $L(\lambda)_{\Gamma}^{\mu}$ and the isomorphisms between them.  For simple Lie algebras this problem was completely solved by Mathieu in \cite{M00}, but the general problem remains open for basic Lie superalgebras.
For modules with a strongly typical central character, this description can be derived from results of Gorelik, Penkov and Serganova \cite{G01, P94, PS92}, however the situation is not surprisingly more difficult when the central character is atypical.

One can further reduce the classification problem to that of classifying ``cuspidal modules'' using Theorem~\ref{cusp}  (Fernando  \cite{F90}; Dimitrov, Mathieu, Penkov \cite{DMP}). A cuspidal module is a simple finite weight $\mathfrak{g}$-module that is not parabolically induced from any proper parabolic subalgebra $\mathfrak{p}\subset\mathfrak{g}$.  These modules can be characterized in terms of their support and in terms of the action of $\mathfrak{g}$.

In this paper, we focus on the family of Lie superalgebras $D(2,1,\alpha)$ which are defined by one complex parameter $\alpha\in\mathbb{C}\setminus\{0,-1\}$ and study their finite weight modules.  For $\mathfrak{g}=D(2,1,\alpha)$, every simple weight module is a finite weight module, and moreover it is bounded!  Indeed, here $\mathfrak{g}_{\bar{0}}=\mathfrak{sl}_2\times\mathfrak{sl}_2\times\mathfrak{sl}_2$, so a simple weight module $V$ for $\mathfrak{g}_{\bar{0}}$ is just the tensor product $V=V_1\otimes V_2\otimes V_3$ of simple $\mathfrak{sl}_2$ weight modules.  Now each simple $\mathfrak{sl}_2$ weight module $V_i$ must have one dimensional weight spaces, since the Casimir element $h^2+2h+fe$ acts by a scalar. Hence, $V$ also has one dimensional weight spaces.  Now any simple weight module of $\mathfrak{g}$ can be realized as the quotient of an induced module $\mbox{Ind}_{\mathfrak{g}_{\bar{0}}}^{\mathfrak{g}}V=U(\mathfrak{g})\otimes_{\mathfrak{g}_{\bar{0}}}V$, where $V$ is a simple $\mathfrak{g}_{\bar{0}}$ weight module.  The claim then follows from the fact that $U(\mathfrak{g}_{\bar{1}})$ is finite dimensional (here $\mbox{dim }U(\mathfrak{g}_{\bar{1}})=256$).

For $\mathfrak{g}=D(2,1,\alpha)$, we prove that every Verma module is bounded and has degree less than or equal to $8$.  It then follows from a result of Grantcharov that simple (finite) weight modules are also bounded of degree less than or equal to $8$. We show that the dimensions of the weight spaces of a cuspidal $\mathfrak{g}$-module are constant on $Q_{\bar{0}}$-cosets and we calculate this degree.  We prove that a cuspidal $\mathfrak{g}$-module has degree $8$ if and only if it is ``typical''. We prove that if $M$ is an ``atypical'' cuspidal $\mathfrak{g}$-module then $2\leq\mbox{deg }M\leq 6$.
We determine the conditions on $\lambda$ and $\Gamma$ that are necessary for $L(\lambda)_{\Gamma}^{\mu}$ to be a cuspidal module. It remains to determine the restrictions on $\mu\in\mathbb{C}^{3}$ and isomorphisms between these modules.

\section{Basic Lie superalgebras}

A simple finite dimensional Lie superalgebra $\mathfrak{g}=\mathfrak{g}_{\bar{0}}\oplus\mathfrak{g}_{\bar{1}}$ is called {\em basic} if  $\mathfrak{g}_{\bar{0}}$ is a reductive Lie algebra,
and there exists an even non-degenerate (symmetric) invariant bilinear form on $\mathfrak{g}$.
These are the Lie superalgebras: $\mathfrak{sl}(m|n)$ for $m\neq n$, $\mathfrak{psl}(n|n)$, $\mathfrak{osp}(m|2n)$, $F(4)$, $G(3)$ and $D(2,1,\alpha)$, $a\in\mathbb{C}\setminus \{0,-1\}$, and finite-dimensional simple Lie algebras.  A basic Lie superalgebra can be represented by a Dynkin diagram, though not uniquely.

Let $\mathfrak{g}$ be a basic Lie superalgebra, and fix a Cartan subalgebra $\mathfrak{h}\subset\mathfrak{g}_{\bar{0}}\subset\mathfrak{g}$. We have a root space decomposition $\mathfrak{g}=\mathfrak{h}\oplus\bigoplus_{\alpha\in\Delta}\mathfrak{g}_{\alpha}$. For each set of simple roots $\Pi\subset\Delta$, we have a corresponding set of positive roots $\Delta^{+}=\Delta^{+}_{\bar{0}}\cup\Delta^{+}_{\bar{1}}$ and a triangular decomposition $\mathfrak{g}=\mathfrak{n}^{-}\oplus\mathfrak{h}\oplus\mathfrak{n}^{+}$. This induces a triangular decomposition of $\mathfrak{g}_{\bar{0}}$, namely, $\mathfrak{g}_{\bar{0}}=\mathfrak{n}_{\bar{0}}^{-}\oplus\mathfrak{h}\oplus\mathfrak{n}_{\bar{0}}^{+}$.
Let $\Pi_{\bar{0}}$ denote the corresponding set of simple roots for $\mathfrak{g}_{\bar{0}}$.
The root lattice of $\mathfrak{g}$ (resp. $\mathfrak{g}_{\bar{0}}$) is defined to be
$Q=\sum_{\alpha\in\Pi} \mathbb{Z} \alpha$ (resp. $Q_{\bar{0}}=\sum_{\alpha\in\Pi_{\bar{0}}} \mathbb{Z} \alpha$).
Let $\rho_{0}=\frac{1}{2}\sum_{\alpha\in\Delta^{+}_{\bar{0}}}\alpha$, $\rho_{1}=\frac{1}{2}\sum_{\alpha\in\Delta^{+}_{\bar{1}}}\alpha$ and $\rho=\rho_{0}-\rho_{1}$. Denote by $U(\mathfrak{g})$ (resp. $U(\mathfrak{g}_{\bar{0}})$) the universal enveloping algebra of  $\mathfrak{g}$ (resp. $\mathfrak{g}_{\bar{0}}$).  See \cite{K77, M12} for definitions and more details.

\subsection{Finite weight modules}

A $\mathfrak{g}$-module $M$ is called a {\em weight module} if it decomposes into a direct sum of weight spaces $M=\oplus_{\mu\in\mathfrak{h}^{*}} M_{\mu}$, where $M_{\mu}= \{m\in M \mid h.m=\mu(h)m \text{ for all }h\in\mathfrak{h}\}$. A weight module $M$ is called {\em finite} if $\mbox{dim }M_{\mu} < \infty$ for all $\mu\in\mathfrak{h}^{*}$. Define the {\em support of $M$} to be the set $$\mbox{Supp }M=\{\mu\in\mathfrak{h}^{*}\mid \mbox{dim }M_{\mu}\neq 0\}.$$

The module $M$ is called {\em bounded} if the exists a constant $c$ such that $\mbox{dim }M_{\mu} < c$ for all $\mu\in\mathfrak{h}^{*}$.  Recall that a $\mathfrak{g}$-module $M=M_{\bar{0}}\oplus M_{\bar{1}}$ is also $\mathbb{Z}/2\mathbb{Z}$-graded.  The {\em degree of M} is defined to be
\begin{equation}\label{degree}
\mbox{deg }M=max_{\mu\in\mathfrak{h}^*} \mbox{dim}(M)_{\mu},
\end{equation}
and we define the {\em graded degree of M} to be $(d_{0},d_{1})$,
where $$d_i=max_{\mu\in\mathfrak{h}^*} \mbox{dim}(M_{\bar{i}})_{\mu} \text{ for }i\in\{0,1\}.$$

\begin{remark}
Clearly, $\mbox{max}\{d_0,d_1\} \leq \mbox{deg }M\leq d_0+d_1$. However, if $M$ is a weight module that can be generated by a single weight vector (i.e. simple or highest weight module), then each weight space of $M$ is either purely even or purely odd, and so in this case $\mbox{deg }M=\mbox{max}\{d_0,d_1\}$.
\end{remark}

Let $M(\lambda)$ denote the Verma module of highest weight $\lambda\in\mathfrak{h}^{*}$ with respect to a set of simple roots $\Pi$, and let $L(\lambda)$ denote its unique simple quotient \cite{M12}. It is clear that $M(\lambda)$ and $L(\lambda)$ are finite weight modules, but they are not always bounded.

For each $\beta\in\Pi$ an odd isotropic root (i.e. $(\beta,\beta)=0$), we have an odd reflection of the set of simple roots $r_{\beta}:\Pi \rightarrow \Pi'$ satisfying $\Pi'=(\Pi\setminus\{\beta\})\cup\{-\beta\}$ \cite{LSS}. Moreover, for a simple highest weight module $L_{\Pi}(\lambda)$ there exists $\lambda'\in\mathfrak{h}^*$ such that $L_{\Pi'}(\lambda')=L_{\Pi}(\lambda)$.  In particular, $\lambda'=\lambda-\beta$ if $(\lambda,\beta)\neq 0$, while $\lambda'=\lambda$ otherwise  \cite{KW}.  Using even and odd reflections one can move between all the different choices of simple roots for a basic Lie superalgebra $\mathfrak{g}$ \cite{S11}.  Moreover, one can move between two different Dynkin diagrams of $\mathfrak{g}$ using only odd reflections.

A simple highest weight module $L(\lambda)$ is called {\em typical} if $(\lambda+\rho,\alpha)\neq 0$ for all $\alpha\in\Delta_{\bar{1}}$, and {\em atypical} otherwise. The notion of typicality is preserved by an odd reflection of the set of simple roots, that is, given an odd reflection $r_{\beta}:\Pi \rightarrow \Pi'$ and $L_{\Pi'}(\lambda')=L_{\Pi}(\lambda)$, then $L_{\Pi'}(\lambda')$ is typical iff $L_{\Pi}(\lambda)$ is typical \cite{KW, S11}.

It was shown by Penkov and Serganova that if $\mathfrak{g}$ is a basic Lie superalgebra, then the category of representations of $\mathfrak{g}$ with a fixed generic typical central character is equivalent to the category of representations of $\mathfrak{g}_{\bar{0}}$ with a certain corresponding central character \cite{P94, PS92}.
This equivalence of categories was extended to representations of $\mathfrak{g}$ with a fixed strongly typical central character by Gorelik in \cite{G01}.  In the case that the root system of $\mathfrak{g}$ is reduced ($\alpha,k\alpha\in\Delta$ implies $k=\pm 1$)  all typical central characters are strongly typical.

\subsection{Cuspidal modules}

A {\em $\mathbb{Z}$-grading} of $\mathfrak{g}$ is a decomposition $ \mathfrak{g}=\oplus_{j\in\mathbb{Z}} \mathfrak{g}(j)$ satisfying $[\mathfrak{g}(i),\mathfrak{g}(j)]\subset\mathfrak{g}(i+j)$ and $\mathfrak{h}\subset\mathfrak{g}(0)$.
A subalgebra $\mathfrak{p}\subset\mathfrak{g}$ is called a {\em parabolic subalgebra} if there exists a $\mathbb{Z}$-grading of $\mathfrak{g}$ such that $\mathfrak{p}=\oplus_{j\geq 0} \mathfrak{g}(j)$.  In this case, $\mathfrak{l}=\mathfrak{g}(0)$ is a {\em Levi subalgebra} and $\mathfrak{n}=\oplus_{j\geq 1} \mathfrak{g}(j)$ is the nilradical of $\mathfrak{p}$.

Let $\mathfrak{p}=\mathfrak{l}\oplus\mathfrak{n}$ be a parabolic subalgebra of $\mathfrak{g}$, and let $S$ be a simple $\mathfrak{p}$-module.  Then $M_{\mathfrak{p}}(S):=\mbox{Ind}_{\mathfrak{p}}^{\mathfrak{g}}S$ has a unique simple quotient $L_{\mathfrak{p}}(S)$.  The module $L_{\mathfrak{p}}(S)$ is said to be {\em parabolically induced}. A simple $\mathfrak{g}$-module is called {\em cuspidal} if it is not parabolically induced from any proper parabolic subalgebra $\mathfrak{p}\subset\mathfrak{g}$.

\begin{theorem}[Fernando  \cite{F90}; Dimitrov, Mathieu, Penkov \cite{DMP}]\label{cusp} Let $\mathfrak{g}$ be a basic Lie superalgebra. Any simple finite weight $\mathfrak{g}$-module is obtained by parabolic induction from a cuspidal module of a Levi subalgebra.
\end{theorem}

This theorem is an important step towards the classification of all simple finite weight modules.
It reduces the general classification problem to that of classifying cuspidal modules.

\begin{theorem}[Fernando \cite{F90}] If $\mathfrak{g}$ is a finite dimensional simple Lie algebra that admits a cuspidal module, then $\mathfrak{g}$ is of type A or C. \end{theorem}

\begin{theorem}[Dimitrov, Mathieu, Penkov \cite{DMP}]
Only the following basic Lie superalgebras admit a cuspidal module: $\mathfrak{psl}(n|n)$, $\mathfrak{osp}(m|2n)$ with $m\leq 6$, $D(2,1,\alpha)$ with $\alpha\in\mathbb{C}\setminus\{0,-1\}$,  $\mathfrak{sl}(n)$, $\mathfrak{sp}(2n)$.
\end{theorem}

The following theorem gives a characterization of cuspidal $\mathfrak{g}_{\bar{0}}$-modules.

\begin{theorem}[Fernando \cite{F90}] Let $\mathfrak{g}_{\bar{0}}$ be a reductive Lie algebra, and let $M$ be a simple finite weight  $\mathfrak{g}_{\bar{0}}$-module.  Then $M$ is cuspidal iff $\mbox{Supp }M$ is exactly one $Q$ coset iff $\mbox{ad }x_{\alpha}$ is injective for all $\alpha\in\Delta$, $x_{\alpha}\in\mathfrak{g}_{\alpha}$.
\end{theorem}

\begin{corollary} Let $\mathfrak{g}_{\bar{0}}$ be a reductive Lie algebra. If $M$ is a cuspidal  $\mathfrak{g}_{\bar{0}}$-module, then $M$ is bounded and $\mbox{dim }M_{\mu}=\mbox{deg }M$ for all $\mu\in\mbox{Supp }M$.\end{corollary}

These conditions are too strict when $\mathfrak{g}$ is a basic Lie superalgebra, and so the following definitions were introduced in \cite{DMP}.
A finite weight module $M$ is called {\em torsion-free} if the monoid generated by $$\mbox{inj }M=\{\alpha\in\Delta_{\bar{0}}\mid  x_{\alpha}\in\mathfrak{g}_{\alpha} \text{ acts injectively on }M\}$$ is a subgroup of finite index in $Q$.
A finite weight module $M$ is called {\em dense} if $\mbox{Supp }M$ is a finite union of $Q'$-cosets, for some subgroup $Q'$ of finite index in $Q$.

\begin{theorem}[Dimitrov, Mathieu, Penkov  \cite{DMP}]\label{equivalence}
Let $\mathfrak{g}$ be a basic Lie superalgebra, and let $M$ be a simple finite weight $\mathfrak{g}$-module.  Then $M$ is cuspidal iff  $M$ is dense iff $M$ is torsion free.
\end{theorem}

The following lemmas are proven in \cite{M00}.

\begin{lemma} Let $\mathfrak{g}_{\bar{0}}$ be a reductive Lie algebra. Any bounded $\mathfrak{g}_{\bar{0}}$-module has finite length. \end{lemma}

\begin{lemma}
Let $\mathfrak{g}$ be a basic Lie superalgebra.  If $M$ is a simple finite weight  $\mathfrak{g}$-module, then for each $\alpha\in\Delta_{\bar{0}}$ the action of $x\in\mathfrak{g}_{\alpha}$ on $M$ is either injective or locally nilpotent.
\end{lemma}

For each $\mathfrak{g}_{\bar{0}}$-module $N$, let
$$\widetilde{N} := \mbox{Ind}_{\mathfrak{g}_{\bar{0}}}^{\mathfrak{g}}(N).$$

\begin{theorem}[Dimitrov, Mathieu, Penkov  \cite{DMP}]\label{decompose} (i) For any finite cuspidal $\mathfrak{g}_{\bar{0}}$-module $N$, the module
$\widetilde{N}$ contains at least one and only finitely many non-isomorphic cuspidal submodules.\\
(ii) For any finite cuspidal $\mathfrak{g}$-module $M$, there is at least one and only finitely
many non-isomorphic cuspidal $\mathfrak{g}_{\bar{0}}$-modules such that $M\subset \widetilde{N}$.
\end{theorem}

\subsection{Coherent families}

Let $C(\mathfrak{h})$ denote the centralizer of $\mathfrak{h}$ in $U(\mathfrak{g}_{\bar{0}})$.  A {\em $(\mathfrak{g},\mathfrak{g}_{\bar{0}})$-coherent family of degree $d$} is a finite weight $\mathfrak{g}$-module $M$ such that $\mbox{dim }M_{\mu}=d$ for all $\lambda\in\mathfrak{h}^{*}$ and the function $\mu\mapsto \mbox{Tr }u|_{M_{\mu}}$ is polynomial in $\mu$, for all $u\in C(\mathfrak{h})$ \cite{Gr03, Gr06, M00}.

\begin{example}[Mathieu \cite{M00}]\label{sl2}
Let $\mathfrak{g}=\mathfrak{sl}_2$ and fix $a\in\mathbb{C}$.  Define a module $V(a)=\oplus_{s\in\mathbb{C}}\mathbb{C}x^{s}$ with the following $\mathfrak{sl}_2$ action.
\begin{align*}
&e\mapsto x^{2}\ d/dx +ax &e.x^{s}=(a+s)x^{s+1}\\
&f\mapsto -d/dx +a(1/x) &f.x^{s}=(a-s)x^{s-1}\\
&h\mapsto 2x\ d/dx  &h.x^{s}=(2s)x^{s}
\end{align*}

For each $a\in\mathbb{C}$, $V(a)$ is a $\mathfrak{sl}_2$-coherent family.  For each $[\mu]\in\mathbb{C}/\mathbb{Z}$ with representative $\mu\in\mathfrak{h}^{*}$,  $$V(a)^{[\mu]}:=\oplus_{n\in\mathbb{Z}}\mathbb{C}x^{\mu+n}$$ is a submodule, which is simple and cuspidal if and only if $\mu\pm a \not\in \mathbb{Z}$.
\end{example}

Let $\Gamma=\{\gamma_1,\ldots,\gamma_k\}\subset\Delta_{\bar{0}}^{+}$ be a set of commuting roots, and for each $\gamma_i\in\Gamma$ choose $f_{i}\in\mathfrak{g}_{-\gamma_i}$. Let $U_{\Gamma}$ be the localization of $U(\mathfrak{g})$ at the set $\{f_{i}^{n}\mid n\in\mathbb{N}, \gamma_i\in\Gamma\}$.
If $\Gamma\subset \mbox{inj }L(\lambda)$, define the {\em localization of $L(\lambda)$ at $\Gamma$} to be the module $L(\lambda)_{\Gamma}:=U_{\Gamma}\otimes_{U(\mathfrak{g})} L(\lambda)$.  Then $L(\lambda)$ is a submodule of $L(\lambda)_{\Gamma}$ and $\mbox{deg }L(\lambda)_{\Gamma}=\mbox{deg }L(\lambda)$.

Now for each $\mu\in\mathbb{C}^{k}$, we define a new module $L(\lambda)_{\Gamma}^{\mu}$ whose underlying vector space is $L(\lambda)_{\Gamma}$, but with a new action of $\mathfrak{g}$ defined as follows.  For $u\in U_{\Gamma}$ and $x\in L(\lambda)_{\Gamma}^{\mu}$, $$u\cdot x:= \Phi^{\mu}_{\Gamma}(u)v,$$ where
$$\Phi^{\mu}_{\Gamma}(u)=\sum_{0\leq i_1,\ldots,i_k} {\mu_1 \choose i_1} \dots {\mu_k \choose i_k} \mbox{ad}(f_1)^{i_1}\dots\mbox{ad}(f_k)^{i_k}(u)f_1^{-i_1}\dots f_k^{-i_k}.
$$
Note that this sum is finite for each choice of $u$.
The module $L(\lambda)_{\Gamma}^{\mu}$ is called the {\em twisted localization of $L(\lambda)$ with respect to $\Gamma$ and  $\mu$}, and it is a $(\mathfrak{g},\mathfrak{g}_{\bar{0}})$-coherent family of degree $d=\mbox{deg }L(\lambda)$.

\begin{theorem}[Mathieu \cite{M00}; Grantcharov \cite{Gr09}]\label{simple}  Let $\mathfrak{g}$ be a basic Lie superalgebra. Each simple finite weight $\mathfrak{g}$-module $M$ is a twisted localization of a simple highest weight module $L_{\mathfrak{b}}(\lambda)$ for some Borel subalgebra $\mathfrak{b}\subset\mathfrak{g}$ and $\lambda\in\mathfrak{h}^{*}$. In particular, $M\cong L_{\mathfrak{b}}(\lambda)_{\Gamma}^{\mu}$ for some $\mu\in\mathbb{C}$ and set of commuting even roots  $\Gamma$.
\end{theorem}

\begin{remark} If $M$ is a cuspidal or bounded module, then $L_{\mathfrak{b}}(\lambda)$ is necessarily bounded.
\end{remark}

\section{The Lie superalgebra $D(2,1,\alpha)$}

For each $\alpha\in\mathbb{C}\setminus\{0,-1\}$, the Lie superalgebra $\mathfrak{g}=D(2,1,\alpha)$ can be realized as a contragredient Lie superalgebra $\mathfrak{g}(A)$ with Cartan matrix \begin{equation}\label{A} A=\left(\begin{array}{ccc}0 & 1 & \alpha\\
1 & 0 & -\alpha-1\\
 \alpha & -\alpha-1 & 0\end{array}\right),\end{equation}
set of simple roots $\Pi=\{\beta_1,\beta_2,\beta_3\}$ with parity $(1,1,1)$,  generating set $$\{e_i\in\mathfrak{g}_{\beta_i}, f_i\in\mathfrak{g}_{-\beta_i}, h_i\in\mathfrak{h}\mid i=1,2,3\}$$ and defining relations \cite{K77}.
Then $\mathfrak{g}_{\bar{0}}=\mathfrak{sl}_2 \times \mathfrak{sl}_2 \times \mathfrak{sl}_2$ is 9-dimensional with \begin{equation*}\Pi_{\bar{0}}=\{\beta_1+\beta_2,\beta_1+\beta_3,\beta_2+\beta_3\},\end{equation*}
and $\mathfrak{g}_{\bar{1}}$ is the 8-dimensional $\mathfrak{g}_{\bar{0}}$-module given by tensoring three copies of the standard representation of $\mathfrak{sl}_2$.
Our choice of $\Pi$ induces a triangular decomposition $\Delta=\Delta^{+}\cup\Delta^{-}$ where

\begin{equation}\label{triangleD} \Delta_{\bar{0}}^{+}=\{\beta_1+\beta_2,\beta_1+\beta_3,\beta_2+\beta_3\} \text{ and } \Delta_{\bar{1}}^{+}=\{\beta_1,\beta_2,\beta_3,\beta_1+\beta_2+\beta_3\}.\end{equation}

\begin{remark}\label{lattice}
For $\mathfrak{g}=D(2,1,\alpha)$, the even root lattice $Q_{\bar{0}}$ is a sublattice of index 2 in $Q$.
\end{remark}

$D(2,1,\alpha)$ has four different Dynkin diagrams.

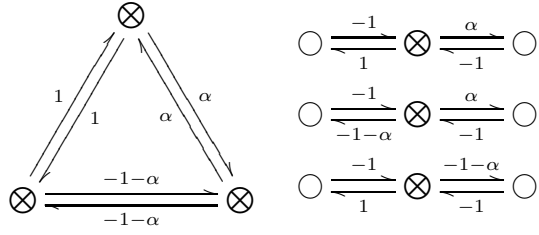
\begin{equation*}\label{diagrams}
\begin{array}{c}\xymatrix{& \OX \AW[ldd]^{1}_{1} \AW[rdd]^{\alpha}_{\alpha} & \\
& & \\ \OX \AW[rr]^{-1-\alpha}_{-1-\alpha} & & \OX } \end{array}
\begin{array}{c}
\xymatrix{ \O \AW[r]^{-1}_{1} & \OX \AW[r]^{\alpha}_{-1} & \O }\\
\xymatrix{ \O \AW[r]^{-1}_{-1-\alpha} & \OX \AW[r]^{\alpha}_{-1} & \O }\\
\xymatrix{ \O \AW[r]^{-1}_{1} & \OX \AW[r]^{-1-\alpha}_{-1} & \O }\end{array}
\end{equation*}

\begin{remark}
The Cartan matrix $A$ in (\ref{A}) is equivalent to the diagram on the left, and corresponds to $\Delta^+$ given in (\ref{triangleD}).  Most of our computations will be with respect to this choice of the set of simple roots, since here $\rho=0$ and since no set of simple roots of $\mathfrak{g}=D(2,1,\alpha)$ contains a set of simple roots for $\mathfrak{g}_{\bar{0}}$.
\end{remark}

\subsection{Finite weight modules for $D(2,1,\alpha)$}\label{cuspD}

In this section we study finite weight modules for the basic Lie superalgebra $D(2,1,\alpha)$.

It was shown in \cite{DMP} that every simple finite weight module is obtained by parabolic induction from a cuspidal module $L_{\mathfrak{p}}(S)$ with $\mathfrak{p}=\mathfrak{l}\oplus\mathfrak{n}$, such that either $\mathfrak{l}$ is a proper reductive subalgebra of $\mathfrak{g}_{\bar{0}}=\mathfrak{sl}_2\times\mathfrak{sl}_2\times\mathfrak{sl}_2$ or $\mathfrak{l}=\mathfrak{g}$. Cuspidal modules for $\mathfrak{sl}_2$ are classified in Example~\ref{sl2}, and all cuspidal modules for $\mathfrak{sl}_2\times\mathfrak{sl}_2$ can be obtained by tensoring two cuspidal $\mathfrak{sl}_2$-modules together, namely, $$V(a_1)^{[\mu_{1}]}\otimes V(a_2)^{[\mu_2]}\text{ with }\mu_i\pm a_i,\not\in\mathbb{Z},\ i=1,2.$$
So it remains to describe the cuspidal modules for $\mathfrak{g}=D(2,1,\alpha)$. The following theorems will help us realize these modules.

\begin{theorem}\label{thmVerma} Let $\mathfrak{g}=D(2,1,\alpha)$.
For each set of positive roots $\Delta^{+}$ and each $\lambda\in\mathfrak{h}^{*}$, the Verma module $M(\lambda)$ is bounded.  $M(\lambda)$ has degree $8$ and  graded degree $(8,8)$.
\end{theorem}

\begin{proof}
The dimensions of the weight spaces of $M(\lambda)$ are given by the coefficients of the character formula $\mbox{ch }M(\lambda)=\frac{e^{\lambda+\rho}}{e^{\rho}R}=e^{\lambda}\frac{R_{1}}{R_{0}}$, where $R_{0}=\Pi_{\alpha\in\Delta_{\bar{0}}^{+}}(1-e^{-\alpha})$ and $R_{1}=\Pi_{\alpha\in\Delta_{\bar{1}}^{+}}(1+e^{-\alpha})$. Since $e^{\rho}R$ is invariant under odd reflections, it is sufficient to prove the theorem with respect to the set of positive roots $\Delta^{+}$ from (\ref{triangleD}).
\begin{equation*}
\begin{aligned}
\mbox{ch }M(\lambda)&=e^{\lambda} \frac{(1+e^{-\beta_1})(1+e^{-\beta_2})(1+e^{-\beta_3})(1+e^{-\beta_1-\beta_2-\beta_3})}
{(1-e^{-\beta_1-\beta_2})(1-e^{-\beta_1-\beta_3})(1-e^{-\beta_2-\beta_3})}\\
&=e^{\lambda}(1+e^{-\beta_1})(1+e^{-\beta_2})(1+e^{-\beta_3})(1+e^{-\beta_1-\beta_2-\beta_3})\cdot\left(\sum_{\substack{k_1,k_2,k_3\in\mathbb{N}:\\
k_1+k_2+k_3\text{ is even}}}e^{-k_1\beta_1-k_2\beta_2-k_3\beta_3}\right)
\end{aligned}
\end{equation*}
So for $m_1,m_2,m_3\in\mathbb{N}$ sufficiently large, the $(\lambda-m_1\beta_1-m_2\beta_2-m_3\beta_3)$ weight space of $M(\lambda)$ is purely even with dimension ${4\choose 0} + {4\choose 2} + {4\choose 4}=8$ when $m_1+m_2+m_3$ is even,  (which corresponds to a choice of an even number of odd roots from $\Delta_{\bar{1}}^{-}$), and it is purely odd with dimension ${4\choose 1} +{4\choose 3}=8$ when $m_1+m_2+m_3$ is odd,   (corresponding to a choice of an odd number of odd roots from $\Delta_{\bar{1}}^{-}$).
\end{proof}

\begin{corollary}\label{corHW}
A highest weight module is bounded and has degree less than or equal to $8$.
\end{corollary}

Since every simple weight module of $D(2,1,\alpha)$ is a finite weight module, we can combine Theorem~\ref{simple} with Corollary~\ref{corHW} to obtain the following.

\begin{theorem}\label{small}
For $\mathfrak{g}=D(2,1,\alpha)$, any simple weight module $M$ is bounded and has degree less than or equal to $8$.
\end{theorem}

\begin{remark}
In Theorem~\ref{small}, we do not assume that $M$ is a highest weight module.
\end{remark}

Let $\Delta^{+}$ be as in (\ref{triangleD}), and let $\alpha_1=\beta_2+\beta_3$, $\alpha_2=\beta_1+\beta_3$, $\alpha_3=\beta_1+\beta_2$, so that $\Pi_{\bar{0}}=\{\alpha_1,\alpha_2,\alpha_3\}$.
Then for each $\alpha_i\in\Pi_{\bar{0}}$, $i=1,2,3$, we have an $\mathfrak{sl}_2$-triple $\{E_{i},F_{i},H_{i}\}$ with $E_{i}\in\mathfrak{g}_{\alpha_{i}}$, $F_{i}\in\mathfrak{g}_{-\alpha_{i}}$ and $H_i=[E_i,F_i]$ satisfying $\alpha_{i}(H_{i})=2$.
For each $\lambda\in\mathfrak{h}^{*}$, $i=1,2,3$, let $\lambda_i=\lambda(h_i)$ and $c_i=\lambda(H_i)$.
Then
\begin{equation}\label{cs}
c_{1}=\frac{\lambda_2+\lambda_3}{-\alpha-1},\ c_2=\frac{\lambda_1+\lambda_3}{\alpha},\ c_3=\lambda_1+\lambda_2.\end{equation}

\begin{theorem}\label{conditions} Let $\mathfrak{g}=D(2,1,\alpha)$ and let $\Delta^+$ be as in (\ref{triangleD}). Then for each $\lambda\in\mathfrak{h}^{*}$, we have that $\mbox{inj }L(\lambda)=\Delta_{\bar{0}}^{-}$ if and only if $c_1,c_2,c_3\not\in\mathbb{Z}_{\geq 1}$ and at most one of $\lambda_1,\lambda_2,\lambda_3$ equals zero.
\end{theorem}

\begin{proof}
Let $\Pi$ denote the set of simple roots of $\Delta^{+}$. Then $F_i$ acts injectively on $L(\lambda)$ if and only if given a set of simple roots $\Pi'$ containing $\alpha_i$ which can be obtained by a sequence of odd reflections from $\Pi$,  we have that $\lambda'(H_i)\not\in\mathbb{Z}_{\geq 0}$, where $\lambda'\in\mathfrak{h}^{*}$ satisfies $L_{\Pi'}(\lambda')=L_{\Pi}(\lambda)$.

Now if $\lambda_i,\lambda_j=0$, then $f_iv=0$ and $f_jv=0$ imply $F_kv=[f_i,f_j]v=0$.  Hence, if $F_1,F_2,F_3$ act injectively on $L(\lambda)$, then at most one of $\lambda_1,\lambda_2,\lambda_3$ equals zero. By reflecting at  $\beta_j$ we get $\Pi'=\{-\beta_j,\alpha_i,\alpha_k\}$ where $i\neq j\neq k$. If $\beta_j$ is a typical root ($\lambda_j\neq0$) then $F_i$ acts injectively iff $(\lambda-\beta_j)(H_i)=c_i-1\not\in\mathbb{Z}_{\geq 0}$, and $F_k$ acts injectively iff $(\lambda-\beta_j)(H_k)=c_k-1\not\in\mathbb{Z}_{\geq 0}$.  Hence, if at least two of  $\lambda_1,\lambda_2,\lambda_3$ are non-zero, it follows that $F_1,F_2,F_3$ act injectively iff $c_1,c_2,c_3\not\in\mathbb{Z}_{\geq 1}$.
\end{proof}

\begin{remark}
Let $\Delta^+$ be as in (\ref{triangleD}), then $L(\lambda)$ is typical if and only if $\lambda_1,\lambda_2,\lambda_3\neq 0$ and $\lambda_1+\lambda_2+\lambda_3\neq 0$.
\end{remark}

\begin{theorem} For $\mathfrak{g}=D(2,1,\alpha)$, suppose that $L(\lambda)$ is a simple highest weight module satisfying $\mbox{inj }L(\lambda)=\Delta_{\bar{0}}^{-}$, then $L(\lambda)= M(\lambda)$ iff $L(\lambda)$ is typical.
\end{theorem}

\begin{proof}
Since each of these properties is a module property that is preserved under odd reflections, it suffices to prove the theorem with respect to $\Delta^{+}$ in (\ref{triangleD}).  By Theorem~\ref{conditions}, we have $\mbox{inj }L(\lambda)=\Delta_{\bar{0}}^{-}$ implies $\lambda(H_i)\not\in\mathbb{Z}_{\geq 1}$ for each $\alpha_i\in\Delta_{\bar{0}}^{+}$. Hence,
$$\frac{2(\lambda+\rho,\alpha_i)}{(\alpha_i,\alpha_i)}=\frac{2(\lambda,\alpha_i)}{(\alpha_i,\alpha_i)}=\lambda(H_i)\quad\not\in\mathbb{Z}_{\geq 1},$$ where the first equality is due to the fact that $\rho=0$ for our choice of $\Delta^{+}$, and the second equality is given by the identification of $\mathfrak{h}$ with $\mathfrak{h}^*$.  The claim now follows from computing the Shapovalov determinant using the formula given in \cite[Section 1.2.8]{GK07}.  Since $(\beta,\beta)=0$ for  $\beta\in\Delta_{\bar{1}}$, we conclude that $M(\lambda)$ is simple if and only if  $(\lambda+\rho,\beta)\neq0$ for all $\beta\in\Delta_{\bar{1}}$.
\end{proof}

\newpage
\subsection{Cuspidal modules for $D(2,1,\alpha)$}

The following theorem gives a characterization of cuspidal modules for $D(2,1,\alpha)$.

\begin{theorem}\label{characterize}
Let $\mathfrak{g}=D(2,1,\alpha)$, and let $M$ be a simple weight  $\mathfrak{g}$-module. The following are equivalent
 \begin{enumerate}
\item $M$ is cuspidal;
\item $\mbox{Supp }M$ is exactly one $Q$ coset;
\item  $x_{\alpha}$ acts injectively for all $\alpha\in\Delta_{\bar{0}}$, $x_{\alpha}\in\mathfrak{g}_{\alpha}$.
\end{enumerate}
\end{theorem}

\begin{corollary}
Let $\mathfrak{g}=D(2,1,\alpha)$. If $M$ is a cuspidal  $\mathfrak{g}$-module, then  $\mbox{dim }M_{\lambda}=\mbox{dim }M_{\mu}$ for all $\lambda - \mu \in Q_{\bar{0}}$.
\end{corollary}

For $\mathfrak{g}=D(2,1,\alpha)$ it follows from \cite{G01}, that for typical central characters we have a $1-1$ correspondence between cuspidal  $\mathfrak{g}$-modules and cuspidal $\mathfrak{g}_{\bar{0}}$-modules.  Here we describe cuspidal $\mathfrak{g}_{\bar{0}}$-modules.

\begin{theorem}
Let $\mathfrak{g}_{\bar{0}}=\mathfrak{sl}_2\times \mathfrak{sl}_2\times \mathfrak{sl}_2$. Then the cuspidal $\mathfrak{g}_{\bar{0}}$-modules are as follows. \begin{equation}\label{V} V_a^{\mu}:=V(a_1)^{[\mu_{1}]}\otimes V(a_2)^{[\mu_2]}\otimes V(a_3)^{[\mu_3]}\quad a,\mu\in\mathbb{C}^{3},\ \mu_i\pm a_i,\not\in\mathbb{Z},\ i=1,2,3\end{equation} Moreover, $\mbox{Supp }V_a^{\mu}=Q+\mu$ and $\mbox{deg }V_a^{\mu}=1$.
\end{theorem}

We have two ways to realize cuspidal modules.  The first method is by decomposing the modules $\widetilde{N}$ appearing in Theorem~\ref{decompose}, since each simple subquotient of $\widetilde{N}$ is cuspidal. It follows from the PBW theorem that $\mbox{deg }\widetilde{V_a^{\mu}}=2^7$, so we see from Theorem~\ref{small} that $\widetilde{V_a^{\mu}}$ is far from simple.  Alternatively, one could determine simplicity conditions for the modules $L(\lambda)_{\Gamma}^{\mu}$ appearing in Theorem~\ref{simple}.

Here we calculate the degree of a cuspidal $\mathfrak{g}$-module $L(\lambda)_{\Gamma}^{\mu}$ using the results from Section~\ref{cuspD} and Shapovalov determinants for the module $M(\lambda)$ \cite{G04, GK07, K78}.

\begin{theorem} Let $\mathfrak{g}=D(2,1,\alpha)$, and suppose that $L(\lambda)_{\Gamma}^{\mu}$  is a (simple) cuspidal $\mathfrak{g}$-module for $\lambda\in\mathfrak{h}^{*}$, $\Delta=\Delta^+\cup\Delta^-$, $\Gamma\subset\Delta^{-}_{\bar{0}}$ and $\mu\in\mathbb{C}^{3}$. Then
\begin{enumerate}
\item $\Gamma=\mbox{inj }L(\lambda)=\Delta_{\bar{0}}^{-}$, and if $\Delta^+$ is as in (\ref{triangleD}) then Theorem~\ref{conditions} applies;
 \item $\mbox{deg }L(\lambda)_{\Gamma}^{\mu}=8$ iff $L(\lambda)$ is typical iff  $L(\lambda)$ has graded degree $(8,8)$;
\item  if $L(\lambda)$ is atypical, then $2\leq \mbox{deg }L(\lambda)_{\Gamma}^{\mu}\leq 6$;
\item if $(\lambda,\beta)=0$ for some simple odd root $\beta$, then $\mbox{deg }L(\lambda)_{\Gamma}^{\mu}\leq 4$.
\end{enumerate}
\end{theorem}

\section*{Acknowledgements}  I would like to thank Maria Gorelik, Ivan Penkov and Vera Serganova for helpful discussions. Supported in part at the Technion by an Aly Kaufman Fellowship.


\begin{thebibliography}{}



\bibitem{DMP} I. Dimitrov, O. Mathieu, I. Penkov {\it On the structure of weight modules}, Trans. Amer.
Math. Soc. 352 (2000), 2857--2869.

\bibitem{F90} S. Fernando, {\it Lie algebra modules with finite-dimensional weight spaces, I}, Trans. Amer.
Math. Soc. 322 (1990), 757--781.

\bibitem{G01} M. Gorelik, {\it Strongly typical representations of the basic classical Lie superalgebras}, J. Amer. Math. Soc. 15 (2001), 167--184.

\bibitem{G04} M. Gorelik, {\it The Kac construction of the centre of $U(\mathfrak{g})$ for Lie superalgebras},
JNMP 11 (2004), 325--349.


\bibitem{GK07} M. Gorelik, V. Kac, {\it On Simplicity of Vacuum modules}, Advances in Math. 211 (2007), 621-677.


\bibitem{Gr03} D. Grantcharov, {\it Coherent families of weight modules of Lie superalgebras and an explicit
description of the simple admissible $\mathfrak{sl}(n+1|1)$-modules}, J. of Algebra 265 (2003), 711--733.

\bibitem{Gr06} D. Grantcharov, {\it On the structure and character of weight modules}, Forum Math. 18 (2006),
933--950.

\bibitem{Gr09} D. Grantcharov, {\it Explicit realizations of simple weight modules of classical Lie superalgebras}, Contemp. Math. 499 (2009) 141--148.

\bibitem{K77} V.G. Kac, {\it Lie superalgebras}, Advances in Math. 26 (1977) 8--96.

\bibitem{K78} V.G. Kac, {\it Representations of classical Lie superalgebras}, Lect. Notes Math. 676, Springer-
Verlag (1978), 597--626.

\bibitem{KW} V. G. Kac and M. Wakimoto, {\it Integrable highest weight modules over affine superalgebras and number theory}, Lie Theory and Geometry, Progress in Math. 123, (1994), 415--456.

\bibitem{LSS} D. Leites, M. Savel'ev, V. Serganova, {\it  Embeddings of Lie superalgebra $\mathfrak{osp}(1|2)$ and
nonlinear supersymmetric equations}, Group Theoretical Methods in Physics, vol. 1 (1986), 377--394.

\bibitem{M00} O. Mathieu, {\it Classification of irreducible weight modules}, Ann. Inst. Fourier 50 (2000),
537--592.

\bibitem{M12} I. Musson,  {\it Lie superalgebras and enveloping algebras}, Graduate Studies in Mathematics, vol. 131, 2012.

\bibitem{P94} I. Penkov, {\it Generic representations of classical Lie superalgebras and their localization},
Monatshefte f. Math. 118 (1994) p.267--313.

\bibitem{PS92} I. Penkov, V. Serganova, {\it Representation of classical Lie superalgebras of type I}, Indag.
Mathem., N.S. 3  (1992), p.419--466.


\bibitem{S11} V. Serganova, {\it Kac-Moody superalgebras and integrability}, in Developments and Trends in Infinite-Dimensional Lie Theory, Progress in Mathematics, Vol. 288, Birkh¨auser, Boston, 2011, 169--218.

\end{thebibliography}
\end{document}